\documentclass[11pt]{article}
\usepackage{epigamath}

\usepackage[notext]{kpfonts}
\usepackage{baskervald}

\setpapertype{A4}

\usepackage[english]{babel}

\usepackage[dvips]{graphicx}     

\title{$\overline{M}_{1,{n}}$ is usually not uniruled in characteristic $p$}
\titlemark{$\overline{M}_{1,{n}}$ is usually not uniruled in characteristic $p$}
\author{\vspace{0cm} Will Sawin}
\authoraddresses{
\authordata{Will Sawin}{\firstname{Will} \lastname{Sawin}\\
\institution{Department of Mathematics, Columbia University, New York, NY }\\
\email{sawin@math.columbia.edu}}
}
\authormark{W. Sawin}
\date{\vspace{-5ex}} 
\journal{\'Epijournal de G\'eom\'etrie Alg\'ebrique} 
\acceptation{Received by the Editors on December 11, 2017, and in final form
on December 19, 2018. \\ Accepted on January 24, 2019.}

\newcommand{\articlereference}{Volume 3 (2019), Article Nr. 3}

\acknowledgement{The author was supported by Dr. Max R\"{o}ssler, the Walter Haefner Foundation and the ETH Zurich Foundation. The author was also in residence at the Mathematical Sciences Research Institute in Berkeley,
California during the Spring 2017 semester and supported by the National Science Foundation under
Grant No. DMS-1440140.}

\usepackage[all]{xy}

\allowdisplaybreaks

\newtheorem{theo}{Theorem}[section]
\newtheorem{cor}[theo]{Corollary}
\newtheorem{prop}[theo]{Proposition}
\newtheorem{lem}[theo]{Lemma}
\newtheorem{remark}[theo]{Remark}
\newtheorem{question}[theo]{Question}
\newtheorem{defi}[theo]{Definition}

\def \Ql {\mathbb Q_\ell}

\def \Frob {\operatorname{Frob}}
\def \Spec {\operatorname{Spec}}

\def \Sym {\operatorname{Sym}}

\def \oo {\operatorname{tdF}}
\def \Fq {\overline{\mathbb F}_q}

\begin{document}


\maketitle



\begin{prelims}


\def\abstractname{Abstract}
\abstract{Using \'{e}tale cohomology, we define a birational invariant for varieties in characteristic $p$ that serves as an obstruction to uniruledness -- a variant on an obstruction to unirationality due to Ekedahl. We apply this to $\overline{M}_{1,n}$ and show that $\overline{M}_{1,n}$ is not uniruled in characteristic $p$ as long as $n \geq p \geq 11$. To do this, we use Deligne's description of the \'{e}tale cohomology of $\overline{M}_{1,n}$ and apply the theory of congruences between modular forms.}

\keywords{Unirational, uniruled, characteristic $p$, moduli of curves, elliptic curves, modular forms, \'etale cohomology}

\MSCclass{14M20}

\vspace{0.15cm}

\languagesection{Fran\c{c}ais}{%

\textbf{Titre. $\overline{M}_{1,n}$ n'est g\'en\'eralement pas unir\'egl\'ee en caract\'eristique $p$} \commentskip \textbf{R\'esum\'e.} Faisant usage de la cohomologie \'etale, nous d\'efinissons un invariant birationnel pour les vari\'et\'es en caract\'eristique $p$ qui constitue une obstruction \`a l'unir\'eglage -- une variante de l'obstruction \`a l'unirationalit\'e formul\'ee par Ekedahl. Nous appliquons ce crit\`ere \`a l'espace $\overline{M}_{1,n}$ et montrons qu'il n'est pas unir\'egl\'e en caract\'eristique $p$ d\`es que $n \geq p \geq 11$. Pour cela, nous utilisons la description de Deligne de la cohomologie \'etale de $\overline{M}_{1,n}$ et nous appliquons la th\'eorie des congruences entre formes modulaires.}

\end{prelims}


\newpage

\setcounter{tocdepth}{1} \tableofcontents

\setcounter{section}{0}

\bigskip
\bigskip
\bigskip

The aim of this paper is to show that $\overline{M}_{1,n}$ is not uniruled in characteristic $p$ whenever $n \geq p-1 \geq 11$ or $p=11$ and $n \geq 11$ (Theorem \ref{main2}). We will also discuss the related concept of unirationality.

For smooth projective varieties $X$ in characteristic zero, $H^0(X,\Omega^i_X)$ is known to be a birational invariant for all $i$ by \cite[Corollary 2 on p. 153]{hironaka64}. If $f: Y \to X$ is a separable morphism, then \[\dim H^0(Y,  \Omega^i_Y )\geq \dim H^0(X,\Omega^i_X ) ,\] so if $X$ is unirational in characteristic zero then $H^0(X,\Omega^i_X)$ vanishes for $i>0$ and if $X$ is separably uniruled then $H^0(X,\Omega^i _X)$ vanishes for $i= \dim X$. 
However $X$ can be inseparably unirational in characteristic $p$ even if $H^0(X,\Omega^i_X)\neq 0$, for instance when $X$ is a supersingular Kummer surface and $i=2$ \cite[Theorem 1.1]{shioda77}.

Using \'{e}tale cohomology, Ekedahl \cite{ekedahl} defined a birational invariant that fixes this problem. This is unsurprising as \'{e}tale cohomology is invariant under inseparable morphisms. Roughly speaking, his invariant measures the multiplicity of all eigenvalues of Frobenius on the \'{e}tale cohomology with compact supports, except for those eigenvalues that ``look like" the eigenvalues of Frobenius with a lower-dimensional variety - because we quotient by the contributions of lower-dimensional varieties, we obtain a birational invariant.

We will apply this invariant to $\overline{M}_{1,n}$. Due to the close relationship between the cohomology of $\overline{M}_{1,n}$ and modular forms, we are able to show that $\overline{M}_{1,n, \mathbb F_p}$ is not unirational, or even uniruled, whenever there is a $p$-ordinary cusp form of level $1$ and weight $k$ for some $k \leq n+1$ (Theorem \ref{main1}). Applying the classical theory of modular forms modulo $p$, we show that $\overline{M}_{1,n, \mathbb F_p}$ is not uniruled for $n \geq p-1 \geq 11$ or $p=11$ and $n\geq 11$ (Theorem \ref{main2}). However, for a given value of $p$, it is likely possible to get a much better value of $n$ by explicitly computing coefficients of modular forms until an ordinary one is found, except for $p=2,3,5,7$ where no such modular forms exist.

This invariant is defined for varieties that are not necessarily smooth or proper. This enables us to avoid the use of stacks when working with $\overline{M}_{1,n}$, though we do not expect there would be any great difficulty in extending these results to Deligne-Mumford stacks. In the smooth and projective case, this obstruction to unirationality is the same as that given by Esnault \cite[Theorem 1.1]{esnault03}, who also showed it was an obstruction even to the weaker property of having a trivial Chow group of zero-cycles.

The unirationality of $\overline{M}_{1,n}$ was completely understood in characteristic zero. It is unirational for $n \leq 10$, because nine general points in $\mathbb P^2$ determine a genus one curve with nine marked points, and the hyperplane class minus twice the first point gives a tenth marked point \cite{chowpointed}. It is not even uniruled for $n >10$, because its Kodaira dimension is zero for $n=11$ and one for $n>11$ \cite[Theorem 3]{BiniFontanari}. In characteristic $p$, the Kodaira dimension is not an obstruction to unirationality, so only the $n \leq 10$ case remains valid, and to my knowledge this is all that was known.

I would like to thank Daniel Litt, Bhargav Bhatt, John Lesieutre, and Remy van Dobben de Bruyn for helpful conversations, and the anonymous referee for useful comments.

We will always take a variety to be a geometrically integral separated scheme of finite type over a field, and the field will always be a finite field $\mathbb F_q$.

\section{A Birational Invariant}

We will actually define an invariant that is a slight variant of Ekedahl's. To make \'{e}tale cohomology into a birational invariant, we simply quotient by the maximum subspace which could come from a variety of lower dimension. We define this using the (geometric) Frobenius action:

\begin{defi} 
{\rm Let $X$ be a variety over $\mathbb F_q$ of dimension $d$. Let $H^i_{\oo} (X)$ be the quotient of $H^i_c(X_{\overline{\mathbb F}_q}, \Ql)$ by the maximal $\Frob_q$-stable subspace on which the eigenvalues of $\Frob_q$ all divide $q^{d-1}$ in the ring of algebraic integers. (The eigenvalues lie in the ring of algebraic integers by \cite[XXI Corollary 5.5.3]{sga7-ii}.)}  \end{defi}

The notation $\oo$ is short for top-dimensional Frobenius - i.e. the part of cohomology where Frobenius does not act by eigenvalues that could come from a lower-dimensional variety via the excision exact sequence.

Ekedahl's invariant \cite[Equation (1)]{ekedahl} can be viewed as the formal sum of the eigenvalues of Frobenius acting on $H^i_{\oo}(X)$ in the free group on the elements of $\overline{\mathbb Q}$, with multiplicity the multiplicity of the eigenvalue. This gives the same information as the characteristic polynomial of Frobenius acting on this vector space.  The vector space carries slightly more information, which might prove useful, but analyzing it is not any more difficult. Indeed, many of our proofs are essentially the same as Ekedahl's, though some are new. Because the proofs are so short, we felt it was worth repeating them in this different context.

\begin{prop}\label{birational} Let $X$ and $Y$ be two varieties that are birationally equivalent. Then

\[ H^i_{\oo}(X) \cong H^i_{\oo}(Y).\]

\end{prop}

\begin{proof} Since every birational equivalence is the composition of an open immersion and the inverse of an open immersion, it suffices to prove this when $Y$ is an open subset $U$ of $X$.
Let $Z$ be the complement of $U$ in $X$. Excision \cite[5.1.16.3]{sga4-3} gives an exact sequence
\[ H^{i-1}_c(Z_{\overline{\mathbb F}_q}, \Ql) \to H^i_c(U_{\overline{\mathbb F}_q}, \Ql) \to H^i_c(X_{\overline{\mathbb F}_q}, \Ql) \to H^i_c(Z_{\overline{\mathbb F}_q}, \Ql) .\]
$Z$ has dimension at most $d-1$, so by \cite[XXI Corollary 5.5.3(iii)]{sga7-ii}, all the eigenvalues of $\Frob_q$ acting on its compactly supported cohomology are algebraic integers dividing $q^{d-1}$. Thus modulo the maximal subspace on which the eigenvalues are algebraic integers dividing $q^{d-1}$, the excision map 
\[H^i_c(U_{\overline{\mathbb F}_q}, \Ql) \to H^i_c(X_{\overline{\mathbb F}_q}, \Ql)\]
is an isomorphism, hence the induced map 
\[ H^i_{\oo}(U) \to H^i_{\oo}(X) \]
is an isomorphism. \hfill $\Box$ \end{proof}

Furthermore:

\begin{prop}\label{covering} Let $Y$ and $X$ be two varieties of the same dimension and let $f: Y \to X$ be a dominant rational map. Then $ H^i_{\oo}(X) $ is a summand of $H^i_{\oo}(Y)$.
\end{prop}

This proposition, and its proof, are a variant of \cite[Theorem 2]{ekedahl}.

\begin{proof} First assume that $f$ is a finite morphism. The composition of the adjunction and trace maps
 \[ \Ql \to Rf_* f^* \Ql  = Rf_! f^* \Ql  \to \Ql\]
 is multiplication by the degree of $f$ \cite[XVI, Proposition 6.2.5]{sga4-3}, hence nonzero, so $\Ql$ is a direct summand of $R f_! f^* \Ql = R f_! \Ql$. Therefore $H^i_c(X_{\Fq}, \Ql)$ is a direct summand of $H^i_c(X_{\Fq}, Rf_!\Ql)$, which equals $H^i_c(Y_{\Fq}, \Ql)$ by the Leray spectral sequence \cite[XVI, Theorem 5.1.8(a)]{sga4-3}, hence $H^i_{\oo}(X)$ is a direct summand of $H^i_{\oo}(Y)$.

Next assume that $f$ is a dominant morphism and not just a rational map. Because every dominant map between varieties of the same dimension is generically finite \cite[Tag 02NX]{stacks-project}, we may pass to an open subset of $X$ where the map is finite, and also pass to the inverse image of that subset in $Y$. Using Proposition \ref{birational}, this does not affect $H^i_{\oo}$. Thus we may reduce to the previous case.

For general $f$, by passing to an open subset of $X$ and $Y$ and using Proposition \ref{birational}, we may assume that $f$ is a morphism and handle it using the previous case. 
\hfill $\Box$
\end{proof}

\begin{prop}\label{geometric} Let $X$ be a variety over $\mathbb F_q$. Then \[H^i_{\oo} (X _{\mathbb F_{q^n}})= H^i_{\oo}(X).\] \end{prop}

\begin{proof} $\Frob_{q^n}$ acts on $H^i_c(X_{\overline{\mathbb F}_{q^n}},\Ql)= H^i_c(X_{\Fq},\Ql)$ by the $n$th power of $\Frob_q$, so its eigenvalues are the $n$th powers of the eigenvalues of $\Frob_q$.  The eigenvalue $\lambda^n$ divides $(q^n)^{d-1}$ in the ring of algebraic integers if and only if $\lambda$ divides $q^d$ in the ring of algebraic integers, because the $n$th root of an algebraic integer is an algebraic integer. So these two vector spaces are manifestly isomorphic. 
\hfill $\Box$
\end{proof}

Using these, we can prove that $H^i_{\oo}(X)$ is an obstruction to unirationality:

\begin{cor}\label{unirational} Let $X$ be a variety over $\mathbb F_q$ of dimension $d$ that is unirational over $\overline{\mathbb F}_q$. Then $H^i_{\oo}(X) =0 $ for $i<2d$. \end{cor}

This is not exactly stated in \cite{ekedahl} but is very similar to \cite[Corollary 3(iii)]{ekedahl}.

\begin{proof}  Because $X$ is unirational over $\overline{\mathbb F}_q$, it is unirational over $\mathbb F_{q^n}$ for some $n$.

When $i=2k$ for $0\leq k \leq d$, $H^i_c(\mathbb P^d_{\Fq}, \Ql)$ is one-dimensional and Frobenius acts on it with eigenvalue $q^k$, and $H^i_c(\mathbb P^d_{\Fq}, \Ql)$ vanishes for all values of $i$ not of this form (see e.g. \cite[Example 16.3]{milneLEC}). For $k<d$ this eigenvalue divides $q^{d-1}$. Hence $H^i_{\oo} (\mathbb P^d)$ is $0$ for $i<2d$. Thus by Proposition \ref{covering} and Proposition \ref{geometric}, the same is true for $H^i_{\oo}(X_{\mathbb F_q^n})$ and $H^i_{\oo}(X)$.
\hfill $\Box$
\end{proof}

\begin{remark}
{\rm For $X$ a smooth projective simply-connected surface over $\mathbb F_q$ satisfying the Tate conjecture, the converse to Corollary \ref{unirational} would follow from Shioda's conjecture \cite[Conjecture on p. 167]{shioda77}. Indeed, recall that all eigenvalues of $\Frob_q$ on $H^2(X_{\Fq},\Ql)$ are algebraic integers of absolute value $q$. Let $\lambda$ be an eigenvalue of $\Frob_q$ on $H^2(X_{\Fq},\Ql)$. If $H^2_{\oo}(X)=0$ then $\lambda$ divides $q$ so $q/\lambda$ is an algebraic integer. Because all Galois conjugates $\sigma(\lambda)$ of $\lambda$ are also eigenvalues, they also have absolute value $|q|$, so $|\sigma{\alpha}|=|q/\sigma(\lambda)|=q/q=1$ for all $\sigma$. Because $\alpha$ is an algebraic integer all whose conjugates have norm $1$, it is a root of unity. Hence every eigenvalue of $\Frob_q$ is $q$ times a root of unity. Let $n$ be the lcm of the orders of these roots of unity. Then the eigenvalues of  $\Frob_{q^n}$ on $H^2(X_{\Fq},\Ql)$ are all equal to $q^n$. Under the Tate conjecture, that implies the cohomology group $H^2(X_{\Fq},\Ql)$ is generated by classes of cycles defined over $\mathbb F_{q^n}$, so $X$ is supersingular in Shioda's sense. Under Shioda's conjecture, because $X$ is supersingular and simply-connected, it is unirational \cite[Conjecture on p. 167]{shioda77}.}
\end{remark}

We can even show $H^d_{\oo}(X)$ is an obstruction to uniruledness:

\begin{cor}\label{uniruled} Let $X$ be a variety over $\mathbb F_q$ of dimension $d$ that is uniruled over $\overline{\mathbb F}_q$. Then $H^d_{\oo}(X) =0 $. \end{cor}

\begin{proof}  Because $X$ is uniruled over $\overline{\mathbb F}_q$, it is uniruled over $\mathbb F_{q^n}$ for some $n$.

Because $X_{\mathbb F_{q^n}}$ is uniruled, it is dominated by $Y \times \mathbb P^1$ for some $Y$. By repeatedly taking a general hyperplane slice of $Y$, we may assume that $Y$ has dimension $d-1$.   By the K\"{u}nneth formula \cite[XVII Theorem 5.4.3]{sga4-3} \[H^d_c(Y_{\Fq} \times \mathbb P^1_{\Fq}) = \sum_{i=0}^d H^i_c(Y_{\Fq},\Ql) \otimes H^{d-i}_c(\mathbb P^1_{\Fq},\Ql).\] Because $H^0_c(\mathbb P^1_{\Fq},\Ql)= \Ql$, $H^2_c(\mathbb P^1_{\Fq},\Ql) =\Ql(-1)$, and all other cohomology groups of $\mathbb P^1$ vanish  (see e.g. \cite[Example 16.3]{milneLEC}) we have \[H^d_c(Y_{\Fq} \times \mathbb P^1_{\Fq}) =H^d_c(Y_{\Fq},\Ql) + H^{d-2}_c(Y_{\Fq},\Ql(-1)).\] By \cite[XXI Corollary 5.5.3(iii)]{sga7-ii}, the eigenvalues of Frobenius on $H^d_c(Y_{\Fq},\Ql)$ divide $q^{d-1}$ in the ring of algebraic integers and the eigenvalues of Frobenius on $H^{d-2}_c(Y_{\Fq},\Ql)$ divide $q^{d-2}$ in the ring of algebraic integers. Hence the eigenvalues of Frobenius on $H^{d-2}_c(Y_{\Fq},\Ql(-1))$, which are $q$ times the eigenvalues of Frobenius on $H^{d-2}_c(Y_{\Fq},\Ql)$, also divide $q^{d-1}$. Thus all eigenvalues of Frobenius on $H^d_c(Y_{\Fq} \times \mathbb P^1_{\Fq})$ divide $q^{d-1}$ in the ring of algebraic integers and so $H^d_{\oo}(Y \times \mathbb P^1,\Ql)=0$.

Thus by Proposition \ref{covering} and Proposition \ref{geometric}, the same is true for $H^d_{\oo}(X_{\mathbb F_q^n})$ and $H^d_{\oo}(X)$.
\hfill $\Box$
\end{proof}

For smooth varieties, we can express $H^i_{\oo}(X)$ in a different way using Poincar\'{e} duality. This connects it to an argument of \cite{esnault03} and implies that $H^i_{\oo}(X)$ is an obstruction, not just to unirationality, but to some weaker conditions, the weakest of which is that $X$ admits a decomposition of the diagonal. 

\begin{prop}\label{smooth} Let \,$X$ \,be \,a \,smooth \,variety \,of \,dimension \,$d$. \,Then \,$H^i_{\oo}(X)$ \,is \,dual \,to \,the \,quotient \,of \,$H^{2d-i}(X_{\Fq}, \Ql)$ by the maximal subspace on which $\Frob_q$ acts by eigenvalues not divisible by $q$. \end{prop}

\begin{proof} By Poincar\'{e} duality \cite[XVIII, Theorem 3.2.5]{sga4-3}, $H^i_c(X_{\Fq},\Ql)$ and $H^{2d-i}(X_{\Fq},\Ql)$ are dual, with $\Frob_q$ acting on the pairing by multiplication by $q^d$. 

 By Jordan normal form, the maximal  $\Frob_q$-stable subspace of $H^i_c(X_{\Fq},\Ql)$ with all eigenvalues of Frobenius dividing $q^{d-1}$ is the complement of the maximal $\Frob_q$-stable subspace of $H^i_c(X_{\Fq},\Ql)$ without any eigenvalues of Frobenius dividing $q^{n-1}$. Thus $H^i_{\oo}(X)$ is isomorphic to the maximal $\Frob_q$-stable subspace of $H^i_c(X_{\Fq},\Ql)$  with eigenvalues $\lambda$ algebraic integers not dividing $q^{d-1}$.

Hence $H^i_{\oo}(X)$ is dual to the maximal $\Frob_q$-stable quotient space of $H^{2d-i}(X_{\Fq},\Ql)$ with eigenvalues $q^d/\lambda$ where $\lambda$ does not divide $q^{d-1}$, i.e. eigenvalues that are not multiples of $q$.

That is the same as the quotient by the maximal subspace with eigenvalues that are multiples of $q$.
\hfill $\Box$ \end{proof}

\begin{remark} 
{\rm Katz's ``Newton above Hodge" conjecture, proved by Mazur \cite[Theorem 1]{mazurnewtonhodge} and Ogus \cite[Theorem 4.5]{ogus}, and the comparison between \'{e}tale and crystalline cohomology together imply that, if $X$ is a smooth projective variety over $\mathbb F_q$ and $H^i(X,\mathcal O_X)=0$, then all eigenvalues of $\Frob_q$ on $H^i(X, \mathbb Q_\ell)$ are divisible by $q$. Thus the invariant $H^{2d-i}_{\oo}(X)$ is closely connected to $H^i(X,\mathcal O_X)$.}
\end{remark}

\begin{prop}\label{esnault} Let $X$ be a smooth projective variety over $\mathbb F_q$. Assume one of the following conditions holds:
\begin{itemize}
\item $Ch^0(X_{\overline{k(X)}})= \mathbb Z$.
\item $X$ admits a decomposition of the diagonal.
\end{itemize}
Then $ H^i_{\oo}(X) =0 $ for $i<2d$ .\end{prop}

\begin{proof} Assume that $X$ admits a decomposition of the diagonal. By \cite[Theorem 1.1]{esnault03}, this implies that the $p$-adic slopes of the eigenvalues of $\Frob_q$ on $H^i(X_{\Fq}, \Ql)$ are at least $1$ for $i>0$. Because every Galois conjugate of an eigenvalue of $\Frob_q$ is also an eigenvalue of $\Frob_q$, and because an algebraic integer all whose Galois conjugates have $p$-adic valuation at least the $p$-adic valuation of $q$ is divisible by $q$, this implies that the eigenvalues must be divisible by $q$.  Hence $H^{2d-i}_{\oo}(X)=0$ for $i>0$ by Proposition \ref{smooth}.

The first condition implies the second condition by \cite[Appendix to Lecture 1, Exercise 1A.4]{bloch80}. 
\hfill $\Box$
\end{proof}

\section{Application to the Moduli Space of Elliptic Curves}

We will apply $H^i_{\oo}$ to the Deligne-Mumford moduli space $\overline{M}_{1,n}$ of genus $1$ curves with $n$ marked points, a singular projective variety. (We will avoid the use of Deligne-Mumford stacks as they are unnecessary for this problem).

It turns out that $H^i_{\oo}(\overline{M}_{1,n})$ is controlled by modular forms:

We say an algebraic integer is prime to $p$ if some coefficient of its characteristic polynomial other than the first is nonzero modulo $p$. We say a Hecke eigenform is ordinary at $p$ if its $p$th Hecke eigenvalue is prime to p.

\begin{prop}\label{deligne} Let $k$ be a natural number with $k\leq n+1$. If the space $S_k(\Gamma(1))$ of cusp forms of weight $k$ and level $1$ contains an eigenform that is ordinary at $p$, then
\[H^{2n+1-k}_{\oo} ( \overline{M}_{1,n, \mathbb F_p} ) \neq 0.\]\end{prop}

First we sketch the proof. Deligne, in \cite{deligne69}, constructed two-dimensional Galois representations associated to modular forms. These representations are defined as subspaces of a certain sheaf cohomology group on a modular curve. The Frobenius eigenvalues of these representations can be related to the Hecke eigenvalues of the modular forms. Using the Leray spectral sequence, we may write the \'{e}tale cohomology of $\overline{M}_{1,n}$ in terms of sheaf cohomology on a modular curve, and by comparing the relevant sheaves we can show that the two-dimensional Galois representations defined by Deligne occur as subquotients of the cohomology of $\overline{M}_{1,n}$. To show that the Frobenius eigenvalues on these Galois representations are prime to $p$, and thus to show via Proposition \ref{smooth} that $H^{i}_{\oo} ( \overline{M}_{1,n, \mathbb F_p} )  \neq  0$, it suffices to show that the Hecke eigenvalues are prime to~$p$.

We need the following lemma, which is surely well-known, though we did not find a suitable reference in the literature:

\begin{lem}\label{obvious} Let $X$ be a variety and $G$ a finite group acting on $X$. Let $X/G$ be the quotient space. Then for all $i$, $H^i_c(X/G, \Ql)= H^i_c(X,\Ql)^G$. \end{lem}

\begin{proof}Let $\pi$ be the projection $X \to X/G$. The morphism $\pi$ is finite, so $R \pi_* \Ql \cong R \pi_! \Ql$, and the fiber of $\pi$ over any point consists of a single orbit. To check that the adjunction map $\Ql \to  (R \pi_* \Ql)^G = (R \pi_! \Ql)^G$ of sheaves on $X/G$ is an isomorphism, it suffices to check it on stalks, so by proper base change \cite[XIII, Corollary 5.2(iii)]{sga4-3} it suffices to check when $X$ is a single orbit and $X/G$ is a single point, where it is obvious. Therefore 

\[ H^i_c(X/G, \Ql)= H^i_c(X/G, R \pi_! \Ql)^G = H^i_c(X/G, (R\pi_! \Ql)^G)=   H^i_c(X,\Ql)^G  \] by the Leray spectral sequence with compact support \cite[XVII, Proposition 5.2.9]{sga4-3}.
\hfill $\Box$\end{proof}

\begin{proof}[Proof of Proposition \ref{deligne}]
Assume $H^{2n+1-k}_{\oo} ( \overline{M}_{1,n, \mathbb F_p} ) = 0$. We will show that all eigenforms in $S_k(\Gamma(1))$ have Hecke eigenvalue divisible by $p$. If they are divisible by $p$, they cannot be prime to $p$, giving the desired contradiction.

Let $m \geq 3$ be prime to $p$. Following Deligne \cite[p. 151]{deligne69}, let $M_{m, \mathbb F_p}$ be the fine moduli space of elliptic curves with full level $m$ structure over $\mathbb F_p$, and let $f_m: E_{m,\mathbb F_p}\to M_{m,\mathbb F_p}$ be the universal family. These spaces exist and they form a smooth proper family of elliptic curves over a smooth scheme, see e.g. \cite[Corollary 4.7.2]{katzmazur}. The group $GL_2(\mathbb Z/m)$ acts on $M_{m,\mathbb F_p}, E_{m,\mathbb F_p}$, and the $n-1$st power $E_{m,\mathbb F_p}^{n-1}$ of $E_{m,\mathbb F_p}^{n-1}$ over $M_{m,\mathbb F_p}$. The quotient $E_{m,\mathbb F_p}^{n-1} / GL_2(\mathbb Z/m)$ loses the level structure and becomes simply the coarse moduli space of elliptic curves with $n-1$ additional marked points (not necessarily distinct), hence is birational to $\overline{M}_{1,n, \mathbb F_p}$, because $M_{1,n,\mathbb F_p}$ is an open subset of both $\overline{M}_{1,n,\mathbb F_p}$ and $E_{m,\mathbb F_p}^{n-1} / GL_2(\mathbb Z/m)$.

As $H^{2n+1-k}_{\oo} ( \overline{M}_{1,n, \mathbb F_p} ) = 0$, by Proposition \ref{birational},  $H^{2n+1-k}_{\oo} ( E_{m,\mathbb F_p}^{n-1} / GL_2(\mathbb Z/m) ) = 0$. By definition, all eigenvalues of $\Frob_p$ on $H^{2n+1-k}_{c} ( E_{m,\overline{\mathbb F}_p}^{n-1} / GL_2(\mathbb Z/m),\Ql ) $ divide $p^{n-1}$ in the ring of algebraic integers.  That cohomology group is the same as the $GL_2(\mathbb Z/m)$-invariant part $H^{2n+1-k}_{c} ( E_{m,\mathbb F_p}^{n-1} ,\Ql )^{ GL_2(\mathbb Z/m)} $ by Lemma \ref{obvious}, so all eigenvalues of $\Frob_p$ on that divide $p^{n-1}$. Because $E_{m,\mathbb F_p}^{n-1}$ is smooth, by Poincar\'e duality \cite[XVIII, Theorem 3.2.5]{sga4-3}, all eigenvalues of Frobenius acting on $H^{k-1} ( E_{m,\overline{\mathbb F}_p}^{n-1} ,\Ql )^{ GL_2(\mathbb Z/m)} $ are divisible by $p$ in the ring of algebraic integers.

By the degeneration of the Leray spectral sequence \cite[Proposition 2.4 and 2.6.4]{deligne68}, 

\[H^{k-1} \left( E_{m,\overline{\mathbb F}_p}^{n-1} ,\Ql \right)^{GL_2(\mathbb Z/m)}\] contains as a summand \[H^1\left(M_{m, \overline{\mathbb F}_p},R^{k-2} \left( f_m^{n-1}\right)_*  \mathbb Q_\ell \right)^{GL_2(\mathbb Z/m)}.\] This, by the K\"{u}nneth formula,  \cite[XVII, Theorem 5.4.3]{sga4-3} contains as a summand  \[H^1\left(M_{m, \overline{\mathbb F}_p}, (R^1 f_{m*} \mathbb Q_\ell)^{\otimes k-2}  \otimes (R^0 f_{m*} \mathbb Q_\ell)^{(n-1)-(k-2)} \right)^{GL_2(\mathbb Z/m)} = H^1\left(M_{m, \overline{\mathbb F}_p}, (R^1 f_{m*} \mathbb Q_\ell)^{\otimes k-2} \right)^{GL_2(\mathbb Z/m)}.\]  This contains as a summand \[H^1\left(M_{m, \overline{\mathbb F}_p}, \Sym^{k-2} (R^1 f_{m*} \mathbb Q_\ell)\right)^{GL_2(\mathbb Z/m)},\] which has as a quotient the parabolic cohomology \[\tilde{H}^1\left(M_{m, \overline{\mathbb F}_p}, \Sym^{k-2} (R^1 f_{m*} \mathbb Q_\ell)\right)^{GL_2(\mathbb Z/m)}.\]  So all eigenvalues of Frobenius on $\tilde{H}^1\left(M_{m, \overline{\mathbb F}_p}, \Sym^{k-2} (R^1 f_{m*} \mathbb Q_\ell)\right)^{GL_2(\mathbb Z/m)}$ are divisible by $p$. Note that we need the assumption $n \geq k+1$ so that $(n-1)-(k-2)\geq 0$.

Deligne \cite[p. 156]{deligne69} defines an action of the Hecke operator $T_p$ on $\tilde{H}^1(M_{m, \overline{\mathbb F}_p}, \Sym^{k-2} (R^1 f_{m*} \mathbb Q_\ell))$, and shows that $T_p$ acts as $F + I_p^* V$ where $F$ is the geometric Frobenius, $V$ is its transpose, and $I_p^*$ is the action of the diagonal element $\begin{pmatrix} p & 0 \\ 0 & p \end{pmatrix} \in GL_2(\mathbb Z/m)$ \cite[Proposition 4.8]{deligne69}. The action of $I_p^*$ factors through the action of $GL_2(\mathbb Z/m)$, so on the $GL_2(\mathbb Z/m)$-invariant subspace, $T_p= F+V$. Note that all eigenvalues of $F$ are divisible by $p$, and $V$ is the transpose of $F$ so all eigenvalues of $V$ are divisible by $p$. Thus because $F$ and $V$ commute \cite[Proposition 4.8(3)]{deligne69}, all eigenvalues of $T_p$ on $\tilde{H}^1(M_{m, \overline{\mathbb F}_p}, \Sym^{k-2} (R^1 f_{m*} \mathbb Q_\ell))^{GL_2(\mathbb Z/m)}$ are divisible by $p$.

By \cite[Corollary 4.2]{deligne69}, the sheaf $R^1 \tilde{a}(M_m,\Sym^{k-2}(R^1f_{m*}\mathbb Q_\ell))$ on $\Spec \mathbb Z[1/m\ell]$ is lisse, and its stalk at the geometric point $\overline{\mathbb F}_p$ is $\tilde{H}^1(M_{m, \overline{\mathbb F}_p}, \Sym^{k-2} (R^1 f_{m*} \mathbb Q_\ell))$. The Hecke operators and $GL_2(\mathbb Z/m)$ also act on this sheaf. Because the sheaf is lisse, the stalks at $\overline{\mathbb F}_p$ and $\overline{\mathbb Q}$ are isomorphic as vector spaces with a $GL_2(\mathbb Z/m)$ and Hecke operator action. Because all Hecke eigenvalues on the $GL_2(\mathbb Z/m)$-invariant part the stalk of $\tilde{H}^1(M_{m, \overline{\mathbb F}_p}, \Sym^{k-2} (R^1 f_{m*} \mathbb Q_\ell))$ at $\overline{\mathbb F}_p$ are divisible by $p$, all Hecke eigenvalues on the $GL_2(\mathbb Z/m)$-invariant part of the stalk of $\tilde{H}^1(M_{m, \overline{\mathbb F}_p}, \Sym^{k-2} (R^1 f_{m*} \mathbb Q_\ell))$ at $\overline{\mathbb Q}$ are divisible by $p$.

Deligne \cite[pp. 154-158]{deligne69} uses several different spaces $W$ of modular forms, indexed by subscripts and superscripts. The upper-left index is the weight, which is equal to the usual weight of modular forms minus two. The lower-left index is the level (where we always take modular forms for the full congruence subgroup $\Gamma(n)$ of that level). The lower-right index defines the coefficient field, where $_\ell$ represents $\mathbb Q_\ell$, $\infty$ represents $\mathbb C$, and an omitted subscript represents $\mathbb Q$. The upper-right superscript may be used to denote the invariants under a group action (a standard notation, which we have used previously in this proof).

By \cite[Corollary 4.2]{deligne69}, the stalk of $\tilde{H}^1(M_{m, \overline{\mathbb F}_p}, \Sym^{k-2} (R^1 f_{m*} \mathbb Q_\ell))$ at $\overline{\mathbb Q}$ is ${}^{k-2}_m W_\ell$, and the $GL_2(\mathbb Z/m)$-invariant part is ${}^{k-2}_m W ^{GL_2(\mathbb Z/m)}_\ell$, which is the tensor product of ${}^{k-2}_m W ^{GL_2(\mathbb Z/m)}$ with $\mathbb Q_\ell$. By definition \cite[p. 158]{deligne69}, ${}^{k-2}_m W ^{GL_2(\mathbb Z/m)}$ is ${}^{k-2}_1 W$, which tensored up with $\mathbb C$ is  ${}^{k-2}_1W_\infty$. So all eigenvalues of $T_p$ on ${}^{k-2}_1W_\infty$ are algebraic integers divisible by $p$. By \cite[Proposition 3.19]{deligne69},  ${}^{k-2}_1W_\infty$ contains $S^k(\Gamma(1))$ as a summand, so all eigenvalues of $T_p$ on $S^k(\Gamma(1))$ are divisible by $p$, as desired.
\hfill $\Box$
 \end{proof}
 
 \begin{theo}\label{main1} Let $k$ be a natural number and let $n \geq k-1$. If the space $S_k(\Gamma(1))$ of cusp forms of weight $k$ and level $1$ contains an eigenform that is ordinary at $p$, then $\overline{M}_{1,n, \mathbb F_p} $ is not uniruled.
 
 \end{theo}
 
 \begin{proof}By Proposition \ref{deligne}, in this case
 \[H^{k-1}_{\oo} ( \overline{M}_{1,{k-1}, \overline{\mathbb F}_p} ) \neq 0.\]
 hence by Corollary \ref{uniruled}, $M_{1,{k-1}, \mathbb F_p}$ is not uniruled.  Because $M_{1,n,\mathbb F_p}$ dominates $M_{1,{k-1}, \mathbb F_p}$ via the map that forgets $n-(k-1)$ points (whose general fiber is an abelian variety and thus contains no rational curves), $M_{1,n,\mathbb F_p}$ is not uniruled. 
 \hfill $\Box$
 \end{proof}
 
 \section{Calculations with Modular Forms}

For any particular $p$ and $k$ we can compute the eigenvalues of the Hecke operator $T_p$ acting on $S_k(\Gamma(1))$ and determine whether any eigenforms are ordinary. For instance it appears that for almost all $p$, the Ramanujan $\Delta$ function, a cusp form of weight $12$, is ordinary. This is true for all $p<10^{10}$ except $p=2,3,5,7,2411,$ and $7758337633$ \cite{lr10}. For all such $p$, $\overline{M}_{1,11,\mathbb F_p}$ is not uniruled. This is best possible in those characteristics as $\overline{M}_{1,10}$ is always unirational (as mentioned earlier). However a method based only on modular forms of weight $12$ is unlikely to work in general, as it is not even known that there are infinitely many $p$ at which $\Delta$ is ordinary. Yet we can still prove a general existence result for an ordinary form, albeit of considerably higher weight than $12$:

\begin{prop}\label{serre existence} Let $p>11$ be a prime. Then the space $S_{p-1}(\Gamma(1))$ of cusp forms of weight $p-1$ and level $1$ contains a $p$-ordinary eigenform.

\end{prop}

\begin{proof} 
In \cite[1.2]{serre73}, Serre defines $\tilde{M}_k$ as the space of power series mod $p$ that are the reductions mod $p$ of modular forms of weight $k$ and level $1$ with integer coefficients. So the lattice inside $S_{p-1}(\Gamma(1))$ consisting of modular forms of weight $p-1$ and level $1$ with integer coefficients maps to $\tilde{M}_{p-1}$, and this map commutes with the Hecke operator $T_p$.

On $\tilde{M}_{p-1}$, the action of $T_p$ is the same as the action of $U$ \cite[2.2]{serre73}, which acts bijectively \cite[Theorem 6(ii)]{serre73}, so all its eigenvalues are nonzero (mod $p$). Hence to show that one of the eigenvalues of $T_p$ on $S_{p-1}(\Gamma(1))$ is nonzero mod $p$, it suffices to show that the image of this map is nonzero. This follows from the existence of any cusp form of weight $p-1$ with integral coefficients, since we can always divide it by $p$ until at least one coefficient is nonzero mod $p$.

The dimension of the space of cusp forms of weight $k$ and level $1$ is $\lfloor k/12\rfloor$, unless $k \equiv 2$ mod $12$, in which case it is $\lfloor k/12 \rfloor-1$. This is easily seen to be greater than $1$ if $k>14$ or $k=12$, and so is greater than $1$ when $k=p-1$ for any primes $p>11$. Since there is an integral basis for the space of cusp forms (for an elementary proof, observe that the basis for all modular forms constructed in \cite[Chapter 5, Proposition 1]{hidalfunction} restricts to a basis of cusp forms), there exists a cusp form of weight $p-1$ with integer coefficients for all such $k$.
\hfill $\Box$
\end{proof}

\begin{theo}\label{main2}Let $p$ be a prime and $n$ a natural number. If $p>11$ and $n\geq p-2$ or $p=11$ and $n\geq 11$, then $\overline{M}_{1,n,\mathbb F_p}$ is not uniruled. \end{theo}

\begin{proof} If  $p>11$, then by Proposition \ref{serre existence}, $S_{p-1}(\Gamma(1))$ contains a $p$-ordinary eigenform. If $p=11$, then $S_{12}(\Gamma(1))$ contains a $p$-ordinary eigenform (the Ramanujan $\Delta$ function) by direct computation (e.g. \cite[abstract]{lr10}). Hence by Theorem \ref{main1}, $\overline{M}_{1,n,\mathbb F_p}$ is not uniruled. 
\hfill $\Box$
\end{proof}

We suggest two open questions that may be interesting:

\begin{question} {\rm For any $p$, is $\overline{M}_{1,n, \mathbb F_p}$ uniruled (or even unirational) for infinitely many $n$?} \end{question}

We have given a negative answer for $p>7$, so the remaining possibilities are $p=2,3,5,7$. A tantalizing fact is that our proof does not just fail by random chance - using the theory of congruences between modular forms, one can show without calculation that for $p<11$ there are no $p$-ordinary cusp forms of level $1$ and any weight $k$ and thus Theorem \ref{main1} will never prove that $M_{1,n,\mathbb F_p}$ is not uniruled.

Indeed, Hida showed that the dimension of the space of $p$-ordinary cusp forms of weight $k$ on $\Gamma_0(p)$ depends only on $p$ mod $k-1$ for $k\geq 2$ \cite[Chapter 7, Theorem 1]{hidalfunction}, and is equal to the dimension of the space of $p$-ordinary cusp forms on $\Gamma(1)$ for $k\geq 3$ \cite[Chapter 7, Proposition 2]{hidalfunction}. Since for each residue class mod $p-1$, we may find a representative in the interval $[3,p+1]$, it follows that if there is any $p$-ordinary cusp form at all, then there is one of weight $\leq p+1$. But the cusp form of smallest weight is $12$, so if there are any $p$-ordinary cusp forms then $p\geq 11$.

If this reason can be turned into a ``geometric" explanation of the $p$-divisibility of the eigenvalues somehow, then that geometric explanation might prove a stronger statement as well, possibly even unirationality.
Alternatively, if a plausible higher-dimensional analogue of Shioda's conjecture \cite[Conjecture on p. 167]{shioda77} is made, this argument could be used to show that $\overline{M}_{1,g}$ is unirational in characteristic $p<11$ for all $g$ conditionally on that conjecture.

\vspace{10pt}

The next question was actually the original question that motivated this work. We were not able to solve it but in considering related problems we were led to the study of $\overline{M}_{1,n}$:

\begin{question} {\rm For any $p$, is $\overline{M}_{g, \mathbb F_p}$ uniruled for infinitely many $g$?} \end{question}

We could also ask about the validity of the characteristic $p$ analogue of Severi's conjecture \cite[p. 880]{severiconjecture}, which would be that $\overline{M}_{g, \mathbb F_p}$ is unirational for all $g$.

Applying the same method to this problem would require intensive study of the non-tautological cohomology of $\overline{M}_g$, so one might seek other methods to resolve the problem. However, we are not aware of any other method for showing non-unirationality of a space in characteristic $p$, nor any method for showing that spaces are in fact uniruled that could conceivably apply to $\overline{M}_g$ for large $g$.

\providecommand{\bysame}{\leavevmode\hbox to3em{\hrulefill}\thinspace}
%
%

\bibliographystyle{amsalpha}
\bibliographymark{References}
\def\cprime{$'$}

\end{document}